\documentclass[11pt]{article}
\usepackage{amsmath,amssymb,amscd,latexsym,amsthm,mathrsfs,amsfonts}
\usepackage[unicode]{hyperref}
\usepackage{hypbmsec,longtable}
\usepackage{graphicx}
\usepackage{fancyhdr}
\ifx\pdfoutput\undefined\else
\hypersetup{
colorlinks=true,
linkcolor=blue,
citecolor=blue,
urlcolor=blue,
filecolor=blue,
bookmarksnumbered=true,
pdfstartview=FitH,
pdfhighlight=/N
}
\fi
\usepackage{pgf,tikz}

\usepackage{array}
\allowdisplaybreaks[1]

\newtheorem{thm}{Theorem}[section] % reset theorem numbering for each chapter

\newtheorem{lemma}[thm]{Lemma} % same for lemma numbers

\theoremstyle{definition}

\newtheorem{defn}[thm]{Definition} % definition numbers are dependent on theorem numbers

\usepackage{cite}

\renewcommand{\author}[1]{\large\rm #1\\ \bigskip}
\newcommand{\address}[1]{{\normalsize\it #1\\}\bigskip}
\renewcommand{\title}[1]{\bigskip\bigskip\Large\bf #1\bigskip\bigskip\\}

\begin{document}

\vglue .3 cm

\begin{center}

\title{A Cayley graph for $F_{2}\times F_{2}$ which is not minimally almost convex.}
\author{Andrew Elvey Price\footnote[1]{website: \hyperlink{https://www.idpoisson.fr/elveyprice/en/}{https://www.idpoisson.fr/elveyprice/en/}\\$~~~~~~~$email: andrew.elvey@univ-tours.fr}}
\address{CNRS, Institut Denis Poisson, Universit\'e de Tours,\\ Parc Grandmont, 37200 Tours, France}

\end{center}
\setcounter{footnote}{0}

\begin{abstract}
We give an example of a Cayley graph $\Gamma$ for the group $F_{2}\times F_{2}$ which is not minimally almost convex (MAC). On the other hand, the standard Cayley graph for $F_{2}\times F_{2}$ does satisfy the falsification by fellow traveler property (FFTP), which is strictly stronger. As a result, any Cayley graph property $K$ lying between FFTP and MAC (i.e., $\text{FFTP}\Rightarrow K\Rightarrow\text{MAC}$) is dependent on the generating set. This includes the well known properties FFTP and almost convexity, which were already known to depend on the generating set as well as Po\'{e}naru's condition $P(2)$ and the basepoint loop shortening property for which dependence on the generating set was previously unknown. We also show that the Cayley graph $\Gamma$ does not have the loop shortening property, so this property also depends on the generating set.
\end{abstract}

\section{Introduction}

%FFTP$\Rightarrow$ 
Over the years numerous Cayley graph properties have been introduced for the sake of better understanding the Cayley graphs with these properties, and their underlying groups. For any such property one can ask the natural question: is the Cayley graph property independent of the generating set? In this article we use the group $F_{2}\times F_{2}$ to show that the answer is no for the falsification by fellow traveller property (FFTP), the loop shortening property and every non-trivial almost-convexity condition. 
In order to prove this, we consider two different presentations for the group $F_{2}\times F_{2}$, whose Cayley graphs exhibit quite different geometry. First we consider the standard presentation 
\[(G,S_{1})=\langle x,y,c,t|xc=cx,yc=cy,xt=tx,yt=ty\rangle.\]
Note that this is a right angled Artin group (with standard presentation), so by \cite[Theorem 3.1]{artinFFTP}, the pair $(G,S_{1})$ satisfies the falsification by fellow traveller property. The other presentation that we will consider is
\[(G,S_{2})=\langle a,b,c,t|ac=ca,bc=cb,act=tac,bct=tbc\rangle,\]
which can be seen to define the same group due to the Tietze transformations $y=bc$ and $x=ac$. %by the Tietze transformations below:
%\begin{align*}
%&\langle a,b,c,d\mid ac=ca,bc=cb,ad=da,bd=db\rangle\\
%=&\langle x,y,a,b,c,d\mid x=ac^{-1},y=bc^{-1},ac=ca,bc=cb,ad=da,bd=db\rangle\\
%=&\langle x,y,a,b,c,d\mid a=xc,b=yc,xcc=cxc,ycc=cyc,xcd=dxc,ycd=dyc\rangle\\
%=&\langle x,y,c,d\mid xcc=cxc,ycc=cyc,xcd=dxc,ycd=dyc\rangle\\
%=&\langle x,y,c,d|xc=cx,yc=cy,xcd=dxc,ycd=dyc\rangle.
%\end{align*}
Our main theorem is that the pair $(G,S_{2})$ is not minimally almost convex. As a corollary, the following properties all depend on the generating set:
\begin{itemize}
	\item The falsification by fellow traveler property
	\item The basepoint loop shortening property
	\item Every non-trivial almost convexity condition. This includes almost convexity, Po\'{e}naru's condition $P(2)$, minimal almost convexity (MAC) and the slightly stronger condition M$'$AC.
\end{itemize}
In the final section we show that $(G,S_{2})$ does not satisfy the loop shortening property, so this property also depends on the generating set.

For groups satisfying FFTP with respect to one generating set but not others, Neumann and Shapiro gave an example in \cite{FFTP} when they introduced the property, while another example with the same underlying group is used in \cite{regularFFTP}. To the author's knowledge, the only previous example for almost convexity was given by Theil \cite{Thiel}. Elder and Hermiller used solvable Baumslag-Solitar groups to show that minimal almost convexity is not a quasi-isometry invariant, nor is M$'$AC \cite{MAC}, so it is not surprising that these both depend on the generating set.  It was previously unknown, however, whether Po\'{e}naru's condition P(2) or either of the loop shortening properties were dependent on the generating set.

\section{The path metric and almost convexity conditions}

The properties considered in this article require us to interpret the Cayley graph $\Gamma$ as a metric space, which we denote $\tilde{\Gamma}$. Edges are then considered to be segments of length 1 between the vertices that they join and the points in this metric space are the vertices of $\Gamma$ as well the interior points of the edges. Assuming $\Gamma$ is a connected graph, any two points $u,v$ in $\tilde{\Gamma}$ are joined by at least one continuous path. We define the distance $d(u,v)$ between $u$ and $v$ to be the length of the shortest such path. This is sometimes called the {\em path metric}. Since the elements of the underlying group $G$ are precisely the vertices of $\Gamma$, these are also points in the metric space $\tilde{\Gamma}$. In particular the identity $1$ of $G$ is a point in $\tilde{\Gamma}$. For convenience, we write $|z|=d(1,z)$ for $z\in\tilde{\Gamma}$. For $g\in G$, the  value $|g|$ can alternatively be seen as the length of the shortest word representing $g$ over the generating set. By the symmetry of the Cayley graph, we have $d(g_{1},g_{2})=d|g_{1}^{-1}g_{2}|$ for $g_{1},g_{2}\in G$.
\begin{defn} The {\em sphere} $Sph(r)$ and the {\em ball} $\overline{B(r)}$ of radius $r$ around $1$ are defined by:
\[Sph(r)=\{z\in\tilde{\Gamma}:|z|=r\}~~~~~~~~\text{and}~~~~~~~~\overline{B(r)}=\{z\in\tilde{\Gamma}:|z|\leq r\}.\]
\end{defn}

Strictly speaking the paths in $\tilde{\Gamma}$ are continuous objects, however since we will only be interested in paths that are in some sense minimal, we can safely assume that all paths considered in this article are made up of a sequence of joined edges, rather than occasionally backtracking in the middle of an edge. The paths can then be defined by the starting point and the word over the generating set representing the path. The Cayley graphs that we consider have no double edges, so can alternatively describe paths between vertices in $\tilde{\Gamma}$ by the sequence of adjacent vertices they pass through. The restriction that a path lies inside the ball $B(r)$ is equivalent to these vertices all lying inside $B(r)$ and no two adjacent vertices in the path both lying on $Sph(r)$, as then the edge between them would be outside $B(r)$. For our purposes this distinction is not important as the Cayley graph $\Gamma$ that we consider is bipartite, so no two adjacent vertices could simultaneously lie on $Sph(r)$.

We are now ready to define {\em almost convexity}.
\begin{defn}Let $G$ be a group and let $S$ be a finite generating set for $G$. Let $r_{0}\in\mathbb{N}$ be a positive integer and let $f:\mathbb{N}\to\mathbb{N}$ be a function. The pair $(G,S)$ is said to satisfy the almost convexity condition $\text{AC}_{f,r_{0}}$ if the following holds: If $r\in\mathbb{Z}$ satisfies $r>r_{0}$ and $u,v\in Sph(r)$ satisfy $d(u,v)\leq 2$, then there is some path $p$  between $u$ and $v$ such that $p$ is contained in $\overline{B(r)}$ and $p$ has length at most $f(r)$.

The pair $(G,S)$ is said to satisfy
\begin{itemize}
\item {\em Almost Convexity (AC)} if it satisfies $\text{AC}_{f,r_{0}}$ for some constant function $f$ and some $r_{0}>0$,
\item {\em Po\'{e}naru's condition $P(2)$} if it satisfies $\text{AC}_{f,r_{0}}$ for some sub-linear function $f$ and some $r_{0}>0$,
\item {\em Minimal Almost Convexity (MAC)} if it satisfies $\text{AC}_{f,r_{0}}$ for some $r_{0}>0$, where $f(x)=2x-1$,
\item {\em M'AC} if it satisfies $\text{AC}_{f,r_{0}}$ for some $r_{0}>0$, where $f(x)=2x-2$.
\end{itemize}\end{defn}

%Using this definition, almost convexity (AC) imposes the strongest possible restriction on $f$, namely that $f$ is a constant function.

Note that the points $u$ and $v$ in the definition above are joined via the origin by a path of length $2r$, so every finitely generated group trivially satisfies any almost convexity condition where $f(r)\geq 2r$. Hence, the requirement that $f(r)=2r-1$ for MAC is the weakest possible non-trivial almost convexity condition. Due to the successive strengthening of these conditions, we have the implication chain
\[\text{AC}\Rightarrow\text{P(2)}\Rightarrow\text{M$'$AC}\Rightarrow\text{MAC}.\]

The last two conditions M$'$AC and MAC were introduced by I. Kapovich in \cite{K2}, where they are called $K(2)$ and $K'(2)$, respectively. These conditions are discussed in greater detail in \cite{MAC}, where Elder and Hermiller show that M$'$AC does not imply $P(2)$. It is still an open question, however, whether the properties M$'$AC and MAC are equivalent.  The property AC was introduced by Cannon in 1987 \cite{cannon}, and it has been widely studied ever since \cite{loop_shortening,BSnotAC,ACnew}.

\section{The pair $(G,S_{2})$ does not satisfy MAC}

 The main theorem in this section is that the group $(G,S_{2})$ with presentation 
\[\langle a,b,c,t|ac=ca,bc=cb,act=tac,bct=tbc\rangle\]
does not satisfy MAC. As a consequence it does not satisfy AC, $P(2)$ or M'AC. As stated in the introduction, a different Cayley graph for the same group satisfies all four of these properties, which proves that these properties are generating set dependent. 

We start by interpreting $(G,S_{2})$ as a HNN extension, as this will allow us to understand the paths in the Cayley graph using Britton's Lemma, which we state below. First let us give a standard definition of HNN extension (See \cite[Chapter IV.2]{combigrouptheory}).
\begin{defn} Let $G$ be a group with presentation $\langle S\mid R\rangle$ and let $A$ and $B$ be isomorphic subgroups of $G$, with $\phi:A\to B$ an isomorphism. We define the HNN extension $G*_{\phi}$ to be the group with presentation
\[\langle S\cup\{t\}\mid R,t^{-1}at=\phi(a)\forall a\in A\rangle.\]
Note that if $S_{A}$ is a finite generating set for $A$ and $S$ and $R$ both finite, we have a finite presentation for $G*_{\phi}$ given by:
\[\langle S\cup\{t\}\mid R,t^{-1}at=\phi(a)\forall a\in S_{A}\rangle.\]
We will call this a standard presentation for the HNN extension.\end{defn}
Consider the group $P=F_{2}\times C_{\infty}=\langle a,b,c|ac=ca,bc=cb\rangle$ and let $H$ be the subgroup generated by $ac$ and $bc$. Then $G$ is the HNN extension $P*_{\phi}$ where $\phi:H\to H$ is the identity, and $(G,S_{2})$ is its standard presentation.
\begin{lemma}(Britton's Lemma) Assume $w\in G$ is given by the product
\[w=p_{0}t^{\epsilon_{1}}p_{1}t^{\epsilon_{2}}\cdots t^{\epsilon_{n}}p_{n},\]
where $n\geq 1$, each $p_{j}\in P$ and each $\epsilon_{j}=\pm1$. Assume further that there is no $1\leq j\leq n$ satisfying both $p_{j}\in H$ and $\epsilon_{j}=-\epsilon_{j+1}$. Then $w\neq 1$.\end{lemma}
Note that if some $j$ satisfied both $p_{j}\in H$ and $\epsilon_{j}=-\epsilon_{j+1}$, then the word defining $w$ could be reduced, as we would have $t^{\epsilon_{j}}p_{j}t^{\epsilon_{j+1}}=p_{j}$, so Britton's Lemma can be interpreted as saying that this is the only situation in which a word can be reduced. In particular, repeating these reductions as long as they are possible allows one to efficiently determine whether a word represents the identity in $G$.

In order to use Britton's Lemma to understand the metric $d$ in $\tilde{\Gamma}$, we first need to understand the subgroup $H$ of $P$, which we do in the following Lemma:

\begin{lemma} $H$ is the subgroup of $G$ consisting of all elements $wc^{k}$, where $w$ is a word in $\{a,b,a^{-1},b^{-1}\}^{*}$ and $k$ is the sum of the exponents in $w$.\end{lemma}

\begin{proof}Let $w$ be a word in $\{a,b,a^{-1},b^{-1}\}^{*}$ and let $k$ be the sum of the exponents in $w$. We will show that $wc^{k}\in H$. Let $w=s_{1}^{p_{1}}s_{2}^{p_{2}}\ldots s_{n}^{p_{n}}$, where each $s_{i}\in\{a,b\}$ and each $p_{i}\in\{-1,1\}$. Then
\[wc^{k}=wc^{p_{1}+\ldots+p_{n}}=(s_{1}c)^{p_{1}}(s_{2}c)^{p_{2}}\ldots(s_{n}c)^{p_{n}}\in H.\]
Now let $h\in H$. We will show that $h=wc^{k}$ for some word $w\in\{a,b,a^{-1},b^{-1}\}^{*}$, where $k$ is the sum of the exponents in $w$. Since $h\in H=\langle ac,bc\rangle$, we can write 
\[h=(s_{1}c)^{p_{1}}(s_{2}c)^{p_{2}}\ldots(s_{n}c)^{p_{n}},\]
where each $s_{i}\in\{a,b\}$ and each $p_{i}\in\{-1,1\}$. Then we have
\[h=s_{1}^{p_{1}}s_{2}^{p_{2}}\ldots s_{n}^{p_{n}}c^{p_{1}+p_{2}+\ldots+p_{n}}=wc^{k},\] 
where $w=s_{1}^{p_{1}}s_{2}^{p_{2}}\ldots s_{n}^{p_{n}}$ and $k$ is the sum of the exponents in $w$.
\end{proof}

Before proving the main theorem, we will prove some Lemmas regarding distances and paths in the metric space $\tilde{\Gamma}$:

\begin{lemma}\label{Lem:sheet_jump} Let $g_{1}\in P$ and let $g_{2}\in tP$, where $P$ is the subgroup of $G$ generated by $a$, $b$ and $c$. Then for any path $p$ in $\tilde{\Gamma}$ from $g_{2}$ to $g_{1}$, there is some $h\in H$ such that $p$ passes through the vertices $ht=th$ and $h$ in that order as well as the edge joining these points.\end{lemma}
\begin{proof}
Let $w_{1}$ and $w_{2}$ be words over the alphabet $\{a,a^{-1},b,b^{-1},c,c^{-1}\}$ representing $g_{1}^{-1}$ and $t^{-1}g_{2}$, respectively, which is possible since $g_{1}^{-1},t^{-1}g_{2}\in P$. Moreover, let $w_{3}$ be a word representing the path $p$. Then $w=tw_{2}w_{3}w_{1}$ is a word representing the identity. Let $\ell$ be the corresponding loop in $\tilde{\Gamma}$. Applying Britton's Lemma, there must be some subword $tw't^{-1}$ or $t^{-1}w't$ of $w$, with $w'$ a word over the alphabet $\{a,a^{-1},b,b^{-1},c,c^{-1}\}$ representing an element of $H$. We can then replace this subword with $w'$ without changing the group element represented. According to Britton's Lemma this can be repeated until the letters $t$ and $t^{-1}$ no longer appear in the word. At the stage just before the initial letter $t$ is removed, it must appear in a subword $tw't^{-1}$, with $w'$ representing an element $h\in H$. So at this point the entire word is $tw't^{-1}w''$, with $w''$ representing $h^{-1}$ and $t^{-1}$ representing a step from $th=ht$ to $h$. Reversing each reduction up to this stage of the algorithm replaces $w'$ or $w''$ with a longer word representing the same element, while the letter $t^{-1}$ still represents a step from $th$ to $h$. Hence this step appears in the original word $w$. Since $w_{1}$ and $w_{2}$ do not contain the letter $t^{-1}$, this occurrence of the letter $t^{-1}$ must appear in the word $w_{3}$ and hence the corresponding step from $th$ to $h$ must appear in the path $p$.
\end{proof}
\begin{lemma}\label{Lem:tlong}For $g\in P$, we have $|gt|=|g|+1$.
\end{lemma}
\begin{proof}
Since $gt$ and $g$ are adjacent, we clearly have $|gt|\leq|g|+1$. Now let $p$ be a path from $g^{-1}$ to $t$ of length $|gt|$. By Lemma \ref{Lem:sheet_jump}, there is some $h\in H$ such that $p$ passes through $h$ and $ht=th$ in that order. Since $p$ is a geodesic, 
\[|gt|=d(g^{-1},h)+1+d(th,t)=d(g^{-1},h)+d(h,1)+1\geq d(g^{-1},1)+1=|g|+1,\]
where the last inequality is by the triangle inequality.\end{proof}
The following lemma describes distances between elements of the subgroup $P$ of $G$. In its proof we will use the quotient $f:G\to P$ defined by $f(a)=a$, $f(b)=b$, $f(c)=c$ and $f(t)=1$. Note that this is well defined since $f(ac)=f(ca)$, $f(bc)=f(cb)$, $f(act)=f(tac)$ and $f(bct)=f(tbc)$.

\begin{lemma}\label{Lem:length_in_P} Let $j,k\in\mathbb{Z}$ with $j\geq0$ and let $w$ be a freely reduced word of length $j$ over the alphabet $\{a,a^{-1},b,b^{-1}\}$. Then $|\overline{w}c^{k}|=j+|k|$.  
\end{lemma}
\begin{proof}
Let $w_{1}$ be the word $w$ followed by $k$ copies of $c$ if $k\geq 0$ and $-k$ copies of $c^{-1}$ if $k<0$.
Since $P=F_{2}\times C_{\infty}$, the word $w_{1}$ is a shortest word for $g=\overline{w}c^{k}$ over the alphabet $\{a,a^{-1},b,b^{-1},c,c^{-1}\}$. Moreover the length of this word is $j+|k|$, so it suffices to prove that there is no shorter word for $g$ including the letters $t$ and/or $t^{-1}$. Suppose $w_{2}$ is a word representing the same group element $g$. Then, since $g\in P$, we have $f(g)=g$, but applying $f$ to $w_{2}$ simply removes every letter $t$ and $t^{-1}$, yielding a shorter word which only contains letters in $\{a,a^{-1},b,b^{-1},c,c^{-1}\}$. Therefore the shorter word, and hence $w_{2}$ must have length at least $j+|k|$, completing the proof of the lemma.
\end{proof}
\begin{center}
\begin{figure}%[htbp]
   \centering
   \includegraphics[width=4in]{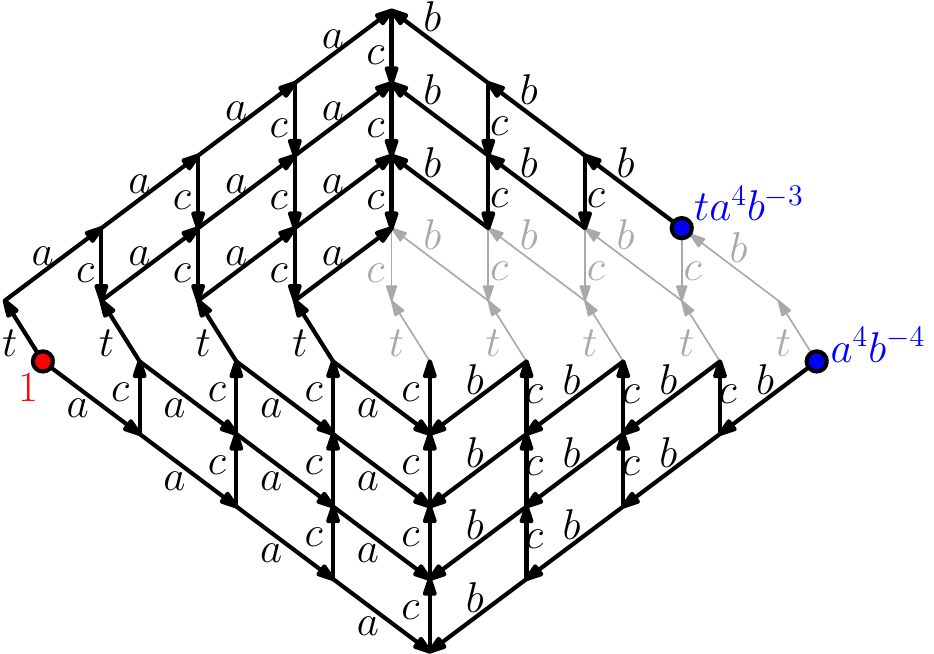} % requires the graphicx package
   \caption{Part of the Cayley graph $\Gamma(G,S_{2})$. The bold vertices and edges are those within the ball of radius 8.}
   \label{fig:ball}
\end{figure}
\end{center}
Finally we are ready to prove our main theorem:
\begin{thm} The pair $(G,S_2)$ does not satisfy MAC.\end{thm}
\begin{proof} For a positive integer $n$, consider the elements $a^{n}b^{-n}$ and $ta^{n}b^{-n+1}$ in $G$ (See Figure \ref{fig:ball}). By Lemmas \ref{Lem:tlong} and \ref{Lem:length_in_P},
\[|a^{n}b^{-n}|=2n=|ta^{n}b^{-n+1}|,\]
so these elements are on the sphere $Sph(2n)$. Moreover, the distance between these points is
\[d(a^{n}b^{-n},ta^{n}b^{-n+1})=d(ta^{n}b^{-n}t^{-1},ta^{n}b^{-n+1})=|tb|=2.\]
So it suffices to prove that if $p$ is a path in $B(2n)$ between $a^{n}b^{-n}$ and $ta^{n}b^{-n+1}$, then $p$ has length at least $4n$. 

Let $p$ be a path in $B(2n)$ from $ta^{n}b^{-n+1}$ to $a^{n}b^{-n}$. By Lemma \ref{Lem:sheet_jump}, there is some $h\in H$ such that $p$ contains $ht$ and $h$ in that order. Moreover, since $p$ is contained in $B(2n)$, we have $|ht|\leq 2n$, so by Lemma \ref{Lem:tlong}, $|h|=|ht|-1\leq 2n-1$. Let $h=wc^{k}$ where $w$ is a freely reduced word over the alphabet $\{a,b,a^{-1},b^{-1}\}$ and $k$ is the sum of the exponents in $w$. Let $j$ be the length of $w$.  Then by Lemma \ref{Lem:length_in_P}, $j+|k|=|h|\leq 2n-1$.

We can bound $d(ht,ta^{n}b^{-n+1})$ using $d(a^{n}b^{-n},h)$, since
\[d(ht,ta^{n}b^{-n+1})=d(h,a^{n}b^{-n+1})\geq d(h,a^{n}b^{-n})-1,\] so the length $|p|$ of $p$ is at least
\[d(a^{n}b^{-n},h)+1+d(ht,ta^{n}b^{-n+1})\geq 2d(a^{n}b^{-n},h),\]
hence it suffices to show that $d(a^{n}b^{-n},h)\geq 2n$.

Now observe that $d(a^{n}b^{-n},h)=|b^{n}a^{-n}wc^{k}|$.
Let $x$ be the number of $a$'s at the start of $w$ and write $w=a^{x} w_{1}$. Then $w_{1}$ has length $j-x$ and does not start with the letter $a$. Also, the sum $k$ of the exponents in $w$ is at least $x-(j-x)$, so $x\leq \frac{j+k}{2}\leq \frac{j+|k|}{2}< n$. Hence, $b^{n}(a^{-1})^{n-x}w_{1}$ is a freely reduced word for $b^{n}a^{-n}\overline{w}$. Therefore, by Lemma \ref{Lem:length_in_P}, we have
\[d(a^{n}b^{-n},h)=\left|b^{n}\left(a^{-1}\right)^{n-x}w_{1}c^{k}\right|=2n-2x+j+|k|\geq 2n,\]
as required.
%so it suffices to show that $2x\leq j+|k|$. Indeed, the sum of the exponents in $w_{1}$ is $k-x$, so it's length, $j-x$ satisfies $j-x\geq |k-x|\geq x-k$, so $2x\leq j+k\leq j+|k|$.

% Therefore,
%\[2n-1\geq|h|=|w|+|k|=x+|w_{1}|+|k|\geq x+|k-x|+|k|\geq 2x,\]
%so $x<n$. Therefore, $b^{n}a^{x-n}w_{1}$ is freely reduced, so 
%\[|b^{n}a^{-n}a^{x}w_{1}|=|b^{n}a^{x-n}w_{1}|=2n-x+|w_{1}|.\]
%Therefore,
%\[d(a^{n}b^{-n},h)=|b^{n}a^{-n}w|+|k|=2n-x+|w_{1}|+|k|\geq 2n-x+x=2n.\]
%Hence, the path $p$ satisfies
%\[|p|\geq d(a^{n}b^{-n},h)+1+d(ta^{n}b^{1-n},ht)\geq d(a^{n}b^{-n},h)+d(ta^{n}b^{-n},ht)=2d(a^{n}b^{-n},h)\geq 4n,\]
%as required.
We have shown that any path contained in $B(2n)$ between $a^{n}b^{-n}$ and $ta^{n}b^{-n+1}$ has length at least $4n$. Since this holds for any positive integer $n$, the pair $(G,S_{2})$ does not satisfy MAC.\end{proof}

\section{Loop shortening properties}

In \cite{loop_shortening} Elder introduced the loop shortening and basepoint loop shortening properties as natural generalisations of the falsification by fellow traveller property (FFTP). Where FFTP gives a simple way to check if a word is a geodesic, each of the loop shortening properties gives a somewhat simple way to check if a word represents the identity in the group.

We define these loop shortening properties below. For convenience we describe the loops only by the vertices that they pass through. As discussed earlier this uniquely defines the loops in our cases as each pair of adjacent vertices in the graph we consider is joined by only one edge. Even without this assumption the definition would make sense if a {\em loop} of vertices is considered to mean a sequence of vertices where each is adjacent to the previous, and the final vertex in the sequence is the same as the starting vertex.

\begin{defn} Let $G$ be a group with finite generating set $S$. $(G,S)$ has the (synchronous) loop shortening property (LSP) if there is a constant $k$ such that for any loop $v_{0},v_{1},\ldots,v_{n}$ in $\Gamma(G,S)$ with $v_{0}=v_{n}$ and $n\geq 1$, there is a shorter loop $u_{0},u_{1},\ldots,u_{m}$ with $u_{m}=u_{0}$ such that $d(u_{j},v_{j})<k$ for each $j\leq m$, and $d(u_{m},v_{j})<k$ for $m\leq j\leq n$. In other words, the paths (synchronously) $k$-fellow travel. \end{defn}

\begin{defn} Let $G$ be a group with finite generating set $S$. $(G,S)$ has the (synchronous) basepoint loop shortening property (BLSP) if there is a constant $k$ such that for any loop $v_{0},v_{1},\ldots,v_{n}$ in $\Gamma(G,S)$ with $n\geq1$, there is a shorter loop $(v_{0}=u_{0}),u_{1},\ldots,u_{m}$ such that $d(u_{j},v_{j})<k$ for each $j\leq m$, and $d(u_{m},v_{j})<k$ for $m\leq j\leq n$. In other words, the paths (synchronously) $k$-fellow travel. \end{defn}

Elder also defined asynchronous versions of these two properties, which he proved to be equivalent to the synchronous versions defined above.

Note that the only difference between these two properties is that for the basepoint loop shortening property, the initial loop is around a basepoint which the shorter loop has to pass through, whereas for the loop shortening property no such restriction is imposed. Hence, it is clear that
\[\text{BLSP}\Rightarrow\text{LSP}.\]
Elder also showed that the basepoint loop shortening property is strictly stronger than almost convexity and strictly weaker than the falsification by fellow traveller property (FFTP), so we have the long implication chain
\[\text{FFTP}\Rightarrow\text{BLSP}\Rightarrow\text{AC}\Rightarrow\text{P(2)}\Rightarrow\text{M$'$AC}\Rightarrow\text{MAC}.\]

Elder asked two questions about the two loop shortening properties. The first is whether they are equivalent, and this remains an open problem. The second is whether either or both of these properties depend on the generating set. We have already shown that the group $G=F_{2}\times F_{2}$ satisfies FFTP with respect to one generating set, but fails MAC with another, which implies that BLSP depends on the generating set. Our final theorem settles the other half of this question, namely that the loop shortening property also depends on the generating set.

\begin{thm}The group $(G,S_{2})$ with presentation 
\[\langle a,b,c,t|ac=ca,bc=cb,act=tac,bct=tbc\rangle\]
does not satisfy the loop shortening property.
\end{thm}
\begin{proof}
Let $k\in\mathbb{Z}_{>0}$. We will show that there is a loop $\ell$ in $\Gamma(G,S_{2})$ such that there is no shorter loop $\ell'$ in $\Gamma(G,S_{2})$ which $k$-fellow travels with $\ell$. Hence this will show that $\Gamma(G,S_{2})$ does not satisfy the loop shortening property. Let $\ell$ be the loop given by the word
\[w=a^{2k}b^{-4k}a^{2k}ta^{-2k}b^{4k}a^{-2k}t^{-1}.\]
The fact that $\overline{w}=1$, follows from $a^{2k}b^{-4k}a^{2k}\in H$, so this does indeed form a loop. Now let $\ell'$ be a loop which $k$-fellow travels with $\ell$. Then we just need to show that the length of $\ell'$ is at least the length of $\ell$, which is $16k+2$. Since the four vertices
\[u_{1}=a^{k},~~u_{2}=a^{2k}b^{-4k}a^{k},~~u_{3}=ta^{2k}b^{-4k}a^{k}~\text{ and }~u_{4}=ta^{k},\]
appear in $\ell$ in that order, there must be vertices $v_{1}$, $v_{2}$, $v_{3}$ and $v_{4}$ appearing in $\ell'$ in that order which satisfy $d(u_{i},v_{i})\leq k$ for each $i\in\{1,2,3,4\}$(See Figure \ref{fig:shorten}). Hence it suffices to prove that
\[d(v_{1},v_{2})+d(v_{2},v_{3})+d(v_{3},v_{4})+d(v_{4},v_{1})\geq 16k+2.\]
For each $i$, let $p_{i}$ be a path of minimal length from $u_{i}$ to $v_{i}$, let $w_{i}$ be the corresponding word and let $g_{i}=\overline{w_{i}}$ be the corresponding group element. So $|w_{i}|\leq k$ and $g_{i}=u_{i}^{-1}v_{i}$. For each $i\in\{1,2,3,4\}$. Let $x_{i}$, $y_{i}$, $z_{i}$ be the sums of the powers of $a$, $b$ and $c$, respectively in the word $w_{i}$. For example this means that $x_{1}$ is the number of occurences of $a$ in $w_{1}$ minus the number of occurences of $a^{-1}$. We will show the following four inequalities, from which the desired result follows:
\[d(v_{1},v_{2})\geq 6k-x_{1}+x_{2}+|y_{1}|+|y_{2}|+|z_1-z_2|,\]
\[d(v_{2},v_{3})\geq 2k+1-x_{2}-x_{3}-y_{2}-y_{3}+z_{2}+z_{3},\]
\[d(v_{3},v_{4})\geq 6k-x_{4}+x_{3}+|y_{4}|+|y_{3}|+|z_4-z_3|,\]
\[d(v_{4},v_{1})\geq 2k+1+x_{1}+x_{4}+y_{1}+y_{4}-z_1-z_4.\]

\begin{figure}[htbp]
   \centering
   \includegraphics[width=5in]{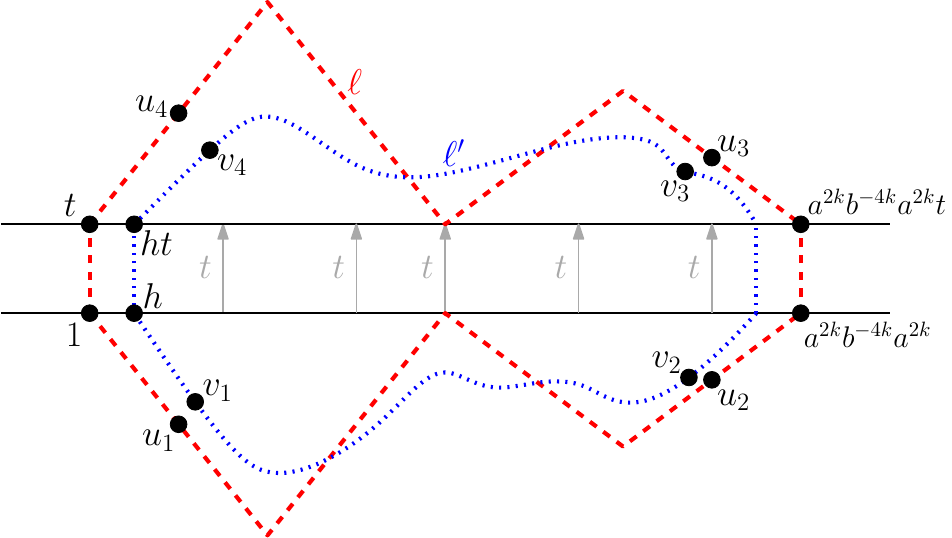} % requires the graphicx package
   \caption{The loop $\ell$ shown by a dashed red line and a $k$-fellow travelling loop $\ell'$ shown by a dotted blue line.}
   \label{fig:shorten}
\end{figure}

We will start by proving the first inequality. Let $w$ be a word of minimal length from $v_{1}$ to $v_{2}$. So \[\overline{w}=g_{1}^{-1}u_{1}^{-1}u_{2}g_{2}=g_{1}^{-1}a^kb^{-4k}a^{k}g_{2}.\]
Consider the quotient map $f:G\to P$ defined above Lemma \ref{Lem:length_in_P} by $f(a)=a$, $f(b)=b$, $f(c)=c$ and $f(t)=1$. Since $c$ commutes with $a$ and $b$, we can write $f(g_{j})=f(\overline{w_{j}})=c^{z_{j}}\overline{r_{j}}$ for each $j$, where $r_{j}$ is a freely reduced word over the alphabet $\{a,b,a^{-1},b^{-1}\}$. The application of $f$ does not change the sum of the powers of $a$ or $b$, nor does the free reduction of the word, so $x_{j}$ and $y_{j}$ are the sums of the powers of $a$ and $b$, respectively, in $r_{j}$. Also note that
\[k\geq|w_{j}|\geq|r_{j}|.\]
Now we return to bounding the distance $d(v_{1},v_{2})$:
\[d(v_{1},v_{2})=|w|\geq |f(\overline{w})|=
|f(g_{1})^{-1}a^{k}b^{-4k}a^{k}f(g_{2})|
=|c^{z_{2}-z_{1}}\overline{r_{1}}^{-1}a^{k}b^{-4k}a^{k}\overline{r_{2}}|.\]
Let $s_{1}$ and $s_{2}$ be freely reduced words for $\overline{r_{1}}^{-1}a^{k}$ and $a^{k}\overline{r_{2}}$ respectively. Since $r_{1}$ has length at most $k$, the word $s_{1}$ must either be the empty word or end with $a$. Similarly, $s_{2}$ must be the empty word or start with $a$. In any case, $s_{1}\left(b^{-1}\right)^{4k}s_{2}$ is a freely reduced word, so by Lemma \ref{Lem:length_in_P}, the length $|w|=|c^{z_{2}-z_{1}}s_{1}\left(b^{-1}\right)^{4k}s_{2}|$ is equal to $|s_{1}|+|s_{2}|+4k+|z_{1}-z_{2}|$. Now, the sums of the powers of $a$ in $s_{1}$ and $s_{2}$ are $k-x_{1}$ and $k+x_{2}$ respectively, and the sums of the powers of $b$ in $s_{1}$ and $s_{2}$ are $-y_{1}$ and $y_{2}$ respectively. Hence, $|s_{1}|\geq k-x_{1}+|y_{1}|$ and $|s_{2}|\geq k+x_{2}+|y_{2}|$. Therefore,
\[d(v_{1},v_{2})\geq |c^{z_{2}-z_{1}}s_{1}b^{-4k}s_{2}|
\geq 6k-x_{1}+x_{2}+|y_{1}|+|y_{2}|+|z_1-z_2|.\]

We can bound $d(v_{3},v_{4})$ similarly as
\[d(v_{3},v_{4})=|g_{4}^{-1}a^{k}b^{-4k}a^{k}g_{3}|\geq|f(g_{4})^{-1}a^{k}b^{-4k}a^{k}f(g_{3})|\]
So the same argument implies
\[d(v_{3},v_{4})\geq 6k-x_{4}+x_{3}+|y_{4}|+|y_{3}|+|z_4-z_3|.\]

Now we will consider the distance $d(v_{1},v_{4})$. Let $w_{5}$ be a word of minimal length from $v_{1}$ to $v_{4}$. Then $w_{1}w_{5}w_{4}^{-1}$ forms a path $p$ from $u_{1}$ to $u_{4}$. Hence by Lemma \ref{Lem:sheet_jump}, there must be some $h\in H$ such that $p$ passes through $h$ and $ht$ in that order. It will be convenient to consider the homomorphism $\pi:P\to \langle e\rangle,$ defined by $\pi(a)=\pi(b)=e$, $\pi(c)=e^{-1}$, which we note is well defined because the definition is consistent with the relations $ac=ca$ and $bc=cb$ of $P$. Then $H$ is the  kernel of $\pi$, so $\pi(h)=1$. Since $\pi(h^{-1}u_{1})=\pi(u_{1})=e^{k}$, the length of any word representing $h^{-1}u_{1}$ must have length at least $k$ i.e., $d(u_{1},h)\geq k$. But the section of $p$ joining $u_{1}$ and $v_{1}$ has length at most $k$. Hence $h$ lies on the section of $p$ between $v_{1}$ and $u_{4}$, Hence $ht$ also lies in this section. Similarly, since $\pi((ht)^{-1}u_{4})=e^{k}$, the points $h$ and $ht$ lie on the section of $p$ between $u_{1}$ and $v_{4}$, so in fact they lie between $v_{1}$ and $v_{4}$. Hence, the distance
\[d(v_{1},v_{4})=d(v_{1},h)+1+d(v_{4},ht).\]
To bound $d(v_{1},h)$, we observe that  $\pi(h^{-1}v_{1})=\pi(v_{1})=\pi(u_{1})\pi(\overline{w_{1}})=e^{k}e^{x_{1}}e^{y_{1}}e^{-z_{1}}=e^{k+x_{1}+y_{1}-z_{1}}$. Therefore,
 \[d(v_{1},h)\geq k+x_{1}+y_{1}-z_{1}.\] Similarly, \[d(v_{4},ht)\geq k+x_{4}+y_{4}-z_{4}.\] Adding these inequalities yields
\[d(v_{4},v_{1})\geq 2k+1+x_{1}+x_{4}+y_{1}+y_{4}-z_1-z_4.\]
Now since $\pi(u_{2})=\pi(u_{3})=e^{-k}$, we can deduce the final inequality,
\[d(v_{2},v_{3})\geq  2k+1-x_{2}-x_{3}-y_{2}-y_{3}+z_{2}+z_{3}.\]
in the same way.
So we now have lower bounds for all four of the distances $d(v_{1},v_{2})$, $d(v_{2},v_{3})$, $d(v_{3},v_{4})$ and $d(v_{4},v_{1})$.
Therefore, the length of the loop $\ell'$ is at least
\begin{align*}&d(v_{1},v_{2})+d(v_{2},v_{3})+d(v_{3},v_{4})+d(v_{4},v_{1})\\
\geq& 6k-x_{1}+x_{2}+|y_{1}|+|y_{2}|+|z_1-z_2|
+2k+1-x_{2}-x_{3}-y_{2}-y_{3}+z_{2}+z_{3}\\
+& 6k-x_{4}+x_{3}+|y_{4}|+|y_{3}|+|z_4-z_3|
+2k+1+x_{1}+x_{4}+y_{1}+y_{4}-z_1-z_4\\
=&16k+2+|y_{1}|+y_{1}+|y_{2}|-y_{2}+|y_{3}|-y_{3}+|y_{4}|+y_{4}\\
+&|z_{1}-z_{2}|-(z_{1}-z_{2})+|z_{4}-z_{3}|-(z_{4}-z_{3})\\
\geq&16k+2,\end{align*}
which is the same as the length of $\ell$. This shows that any loop $\ell'$ that $k$-fellow travels with $\ell$ has length at least that of $\ell$. Since this situation occurs for any integer $k>0$, the pair $(G,S_{2})$ does not enjoy the loop shortening property.\end{proof}

\section{Acknowledgements}
We would like to thank the anonymous referee whose comments greatly improved the readability of the paper.

\end{document}